\documentclass[11pt,a4paper]{amsart}
\usepackage{subfiles}
\usepackage{lipsum}
\usepackage{hyperref}

\hypersetup{urlcolor=blue, citecolor=blue, linkcolor=blue, colorlinks=true}

\usepackage{enumitem}
\usepackage{amsmath}
\usepackage{amssymb}
\usepackage{mathtools}
\usepackage{amsfonts}
\usepackage{microtype}
\usepackage{multicol}
\usepackage{subcaption}
\usepackage{xcolor}
\usepackage{stackengine,scalerel}
\usepackage{tikz}
\usepackage{csquotes}
\usepackage{tikz-cd}
\usepackage{amsthm}
\usepackage{tcolorbox}
\usepackage[capitalise]{cleveref}
\usepackage[utf8]{inputenc}
\usepackage[english]{babel}

\DeclareMathAlphabet{\mathpzc}{OT1}{pzc}{m}{it}

\usepackage{newunicodechar}
\usepackage[style=alphabetic,backend=bibtex,maxnames=5,maxalphanames=5,language=english,doi=false,isbn=false,url=true]{biblatex}
\bibliography{literature.bib}
  
\renewbibmacro{in:}{}
\renewbibmacro*{issue+date}{%
  \ifboolexpr{not test {\iffieldundef{year}} or not test {\iffieldundef{issue}}}
    {\printtext[parens]{%
       \iffieldundef{issue}
         {\usebibmacro{date}}
         {\printfield{issue}%
          \setunit*{\addspace}%
          \usebibmacro{date}}}}
    {}%
  \newunit}

\let\oldtocsection=\tocsection
\let\oldtocsubsection=\tocsubsection
\let\oldtocsubsubsection=\tocsubsubsection
\renewcommand{\tocsection}[2]{\hspace{0em}\oldtocsection{#1}{#2}}
\renewcommand{\tocsubsection}[2]{\hspace{1em}\oldtocsubsection{#1}{#2}}
\renewcommand{\tocsubsubsection}[2]{\hspace{2em}\oldtocsubsubsection{#1}{#2}}

\newtheorem{bigthm}{Theorem}

\newtheorem{thm}{Theorem}[section]
\newtheorem{lem}[thm]{Lemma}
\newtheorem{prop}[thm]{Proposition}
\newtheorem{cor}[thm]{Corollary}
\newtheorem{question}[thm]{Question}

\theoremstyle{definition}
\newtheorem{dfn}[thm]{Definition}

\theoremstyle{remark}

\newtheorem{rem}[thm]{Remark}

\setenumerate[1]{leftmargin=*,labelindent=0pt,label=(\roman*),}

\newcommand{\pref}[2]{\hyperref[#2]{#1 \ref*{#2}}}
\usepackage{todonotes}

\newcommand{\Spin}{\ensuremath{\operatorname{Spin}}}

\newcommand{\diffu}{\ensuremath{\operatorname{Diff}}}
\newcommand{\diff}{\ensuremath{\operatorname{Diff}}}
\newcommand{\blockdiff}{\ensuremath{\widetilde{\operatorname{Diff}}}}
\newcommand{\diffs}{\ensuremath{\operatorname{Diff}^{\scaleobj{0.8}{\Spin}}}}
\newcommand{\maps}{\ensuremath{\operatorname{maps}}}

\newcommand{\sign}{\ensuremath{\operatorname{sign}}}
\newcommand{\sing}{\ensuremath{\operatorname{Sing}}}
\newcommand{\colim}{\ensuremath{\operatorname{colim}}}
\newcommand{\too}{\longrightarrow}
\newcommand{\congarrow}{\overset{\cong}\longrightarrow}
\newcommand{\embeds}{\hookrightarrow}
\newcommand{\actson}{\curvearrowright}
\DeclarePairedDelimiter{\scpr}{\langle}{\rangle}

\DeclarePairedDelimiter\floor{\lfloor}{\rfloor}

\newcommand{\haut}{\ensuremath{\operatorname{hAut}}}
\newcommand{\blockhaut}{\ensuremath{\widetilde{\operatorname{hAut}}}}

\newcommand{\calR}{\mathcal{R}}
\newcommand{\calA}{\mathcal{A}}

\newcommand{\calN}{\mathcal{N}}
\newcommand{\calS}{\mathcal{S}}
\newcommand{\calL}{\mathcal{L}}

\newcommand{\bbC}{\mathbb{C}}
\newcommand{\bbZ}{\mathbb{Z}}
\newcommand{\bbQ}{\mathbb{Q}}
\newcommand{\bbR}{\mathbb{R}}

\newcommand{\hp}[1]{\mathbb{HP}^{#1}}

\newcommand{\ch}{{\mathrm{ch}}}
\newcommand{\ph}{{\mathrm{ph}}}
\newcommand{\ko}{{\mathrm{KO}}}
\newcommand{\pt}{{\mathrm{pt}}}

\newcommand{\psec}{{\mathrm{Sec}>0}}

\newcommand{\prc}{{\mathrm{Ric}>0}}

\newcommand{\psc}{{\mathrm{scal}>0}}

\begin{document}
\author{Georg Frenck}
\email{\href{mailto:georg.frenck@kit.edu}{georg.frenck@kit.edu}}
\email{\href{mailto:math@frenck.net}{math@frenck.net}}
\address{Institut f\"ur Algebra und Geometrie, Englerstr.~2, 76131 Karlsruhe, Germany}

\subjclass[2010]{53C21,55R40, 57R20, 57R22, 58D17, 58D05.}

\thanks{I am supported by the DFG (German Research Foundation) -- 281869850 (RTG 2229).}

\title[Diffeomorphisms and positive curvature]{Diffeomorphisms and positive curvature}

\begin{abstract} 
	We prove the existence of elements of infinite order in the homotopy groups of the spaces $\calR_\prc(M)$ and $\calR_\psec(M)$ of positive Ricci and positive sectional curvature, provided that $M$ is high-dimensional and $\Spin$, admits such a metric and has a non-vanishing rational Pontryagin class. 
\end{abstract}

\maketitle

\section{Introduction}
\noindent For a given closed oriented manifold $M$ let $\calR_\psec(M)$ and $\calR_\prc(M)$ denote the spaces of Riemannian metrics of positive sectional or positive Ricci curvature. In contrast to the space $\calR_\psc(M)$ of positive scalar curvature metrics, rather little is known about the topology of those two spaces; especially if one is interested in rational homotopy groups.

In order to state our main result, let $\diffu(M)$ denote the group of orientation preserving diffeomorphisms of $M$ and let $\diffu(M,D)\subset\diffu(M)$ denote the subgroup of those diffeomorphisms that fix an embedded disk $D\subset M$ point-wise. Furthermore, let $\calR_C(M)\subset\calR_\psc(M)$ be a $\diff(M)$-invariant subset.

\begin{bigthm}\label{thm:main}
	Let $M^d$ be closed, simply connected $\Spin$-manifold that has at least one non-vanishing rational Pontryagin class and let $g\in\calR_C(M)$. Let $k\ge1$ be such that $(d+k)$ is divisible by $4$ and $k\le \min(\frac{d-1}{3},\frac{d-5}{2})$. Then map
	\begin{center}
	\begin{tikzpicture}
		\node (0) at (0,0.7){$\pi_{k-1}(\diff(M,D))\otimes\bbQ$};
		\node (1) at (5,0.7){$\pi_{k-1}(\calR_C(M))\otimes\bbQ$};
		
		\draw[->] (0) to (1);
	\end{tikzpicture}
	\end{center}
	induced by the orbit map $f\mapsto f^*g$ is nontrivial. 
	
	In particular: $\pi_{k-1}(\calR_\prc(M))\otimes\bbQ\not=0\not=\pi_{k-1}(\calR_\psec(M))\otimes\bbQ$, provided these spaces are non-empty.
\end{bigthm}

\begin{rem}[State of the art concerning $\calR_\psec(M)$]
	To the best of the author's knowledge, the following is all that is known about the homotopy type of $\calR_\psec(M)$:
	\begin{enumerate}
		\item Kreck--Stolz have shown in \cite{kreckstolz}, that there exists manifolds $M$ such that $\calR_{\psec}(M)$ is not connected. Their result remains true fore the quotient $\calR_\psec(M)/\diff(M)$, so those components do not originate from the orbit map.
		\item Crowley--Schick--Steimle have shown in \cite{crowleyschicksteimle} that for every manifold $M$, the image of the orbit map $\pi_{k-1}(\diff(M))\to\pi_{k-1}(\calR_\psec(M))$ contains $\bbZ/2$ as a subgroup provided $d+k\equiv1,2\;(8)$ and $d\ge6$. This extends earlier results of Hitchin \cite{hitchin_spinors} and Crowley--Schick \cite{crowleyschick}.
		\item Krannich--Kupers--Randal-Williams have proven the special case $M=\hp2$, $k=4$ of \pref{Theorem}{thm:main} in \cite{KKRW}. Their proof delivers an excellent blueprint for our generalisations. Like in loc.cit., we will construct an $M$-bundle $E\to S^k$ with a $\Spin$-structure on the vertical tangent bundle and non-vanishing $\hat\calA$-genus. Since \cite{KKRW} is written rather densely, we chose to give a more detailed account of their argument in \pref{Section}{sec:prelim} before we go on to proving \pref{Theorem}{thm:main} in \pref{Section}{sec:main}.
	\end{enumerate}	 
\end{rem}

\begin{rem}\leavevmode
\begin{enumerate}
	\item The bound on $k$ can be improved if $M$ is even-dimensional and $\ell$-connected. In this case \pref{Theorem}{thm:main} holds true for $k\le \min(d-4,2\ell-1)$ (cf. \pref{Remark}{rem:morlet}).
	\item As pointed out in \cite{KKRW} this answers a question of Schick \cite[p. 30]{owl} and provides many examples for manifolds to which \cite[Theorem 2.1]{bew} is applicable. 
\end{enumerate}
\end{rem}

\noindent According to \cite{ziller} there are the only known examples of positively manifolds curved in dimensions $4k+3$ for $k\ge2$ are spheres. Also, all $7$-dimensional examples have finite fourth cohomology (cf. \cite{eschenburg, goette_pk, berger_space}). Therefore, a positive answer to the following question would yield the first example of a manifold that admits infinitely many pairwise non-isotopic metrics of positive sectional curvature.

\begin{question}
	Is there a positively curved manifold of dimension $4k+3$, $k\ge1$ with a non-vanishing rational Pontryagin class?
\end{question}  

\subsection*{Acknowledgements} I would like to thank Jens Reinhold for comments on an earlier draft and Bernhard Hanke and Jost Eschenburg for valuable remarks.

\section{Setting the stage}\label{sec:prelim}

\noindent Let $M$ be a closed oriented manifold of dimension $d$ and let $\diff(M)$ denote the group of orientation preserving diffeomorphisms of $M$. We denote by $B\!\diff(M)$ the classifying space for fibre bundles with fibre $M$ and oriented vertical tangent bundle.


\subsection{Block diffeomorphisms}
In this subsection we give a short overview of block diffeomorphisms and we explain how to compare them to diffeomorphisms. For $p\ge0$ let $\Delta^p$ denote the standard topological $p$-simplex. 

\begin{dfn}
	 A \emph{block diffeomorphism} of $\Delta^p\times M$ is a diffeomorphism of $\Delta^p\times M$ that for each face $\sigma\subset\Delta^p$ restricts to a diffeomorphism of $\sigma\times M$. 
\end{dfn}

\noindent The set of all block diffeomorphisms forms a semisimplicial group denoted by $\blockdiff_{\bullet}(M)$ whose $p$-simplices are the block diffeomorphisms of $\Delta^p\times M$. The space $\blockdiff(M)$ of block bundles is defined as the geometric realisation of $\blockdiff_\bullet(M)$ and the associated classifying space is denoted by $B\blockdiff(M)$.

If we consider the semisimplicial subgroup $\diff_\bullet(M)$ of those block diffeomorphisms that commute with the projection $\Delta^p\times M\to \Delta^p$, we precisely get the $p$-simplices of the singular semisimplicial group $\sing_\bullet\diff(M)$. We have an inclusion $\sing_\bullet\diff(M)\subset\blockdiff_\bullet(M)$ and since the geometric realisation of $\sing_\bullet(X)$ is homotopy equivalent to $X$ for any space $X$ (\cite[pp. 8]{handbook_at}), we get an induced map
\[B\!\diff(M)\too B\blockdiff(M).\]

Next, let $\haut(M)$ denote the group-like\footnote{A topological space $X$ is called \emph{group-like} if $\pi_0(X)$ is a group} topological monoid of (orientation preserving) homotopy equivalences of $M$ with classifying space $B\!\haut(M)$. Again, let $\blockhaut(M)$ be the realisation of the semisimplicial group of block homotopy equivalences defined analogously with $B\!\haut(M)$ and $B\blockhaut(M)$ the corresponding classifying spaces. By \cite[Thm 6.1]{dold_partitions} $\haut(M)$ and $\blockhaut(M)$ are homotopy equivalent. Consider the following maps induced by inclusions:
\[B\blockdiff(M) \to B\blockhaut(M)\simeq B\!\haut(M)\qquad B\!\diff(M)\to B\!\haut(M)\]
and let $\haut(M)/\blockdiff(M)$ and $\haut(M)/\diff(M)$ denote the respective homotopy fibres. Note that $\haut(M)/\diff(M)$ classifies $M$-bundles that are homotopy equivalent to the trivial bundle through a homotopy that commutes with the projection of the bundle, i.e. \emph{fibre homotopy trivial} $M$-bundles. We have the following comparison result which easily follows from \cite[Corollary D]{BurgheleaLashof}. 
\begin{lem}
	If $k\le\min(\frac{d-1}{3},\frac{d-5}{2})$ then the map
	\[\pi_k\left(\frac{\haut(M)}{\diff(M)}\right)\left[\frac12\right] \too \pi_k\left(\frac{\haut(M)}{\blockdiff(M)}\right)\left[\frac12\right]\]
	is (split-)surjective.
\end{lem} 
\begin{proof}
	 By \cite[Corollary D]{BurgheleaLashof}  there exists a space $\calS$ such that for $k\le\min(\frac{d-1}{3},\frac{d-5}{2})$ we have
	 \begin{align*}
	 	\pi_{k}&\left(\frac{\haut(M)}{\diff(M)}\right)\left[\frac12\right]\cong\pi_{k-1}(\Omega\left(\frac{\haut(M)}{\diff(M)}\right))\left[\frac12\right]\\
			&\cong \pi_{k-1}(\Omega\left(\frac{\haut(M)}{\blockdiff(M)}\right)\times\Omega\calS)\left[\frac12\right]\twoheadrightarrow \pi_{k}\left(\frac{\haut(M)}{\blockdiff(M)}\right)\left[\frac12\right]\qedhere
	\end{align*}
\end{proof}

\noindent Therefore, an element of $\pi_k(\haut(M)/\blockdiff(M))\otimes\bbQ$ yields an $M$-bundle $E\to S^k$ that is fibre homotopy trivial, provided that the dimension of $M$ is high enough. The advantage of working with $\haut(M)/\blockdiff(M)$ instead of $\haut(M)/\diff(M)$ stems from the fact, that the former is accessible through surgery theory as we will review in the succeeding section.

\begin{rem}\label{rem:morlet}
	Another approach to compare $B\!\diff(M)$ and $B\blockdiff(M)$ is by using Morlet's lemma of disjunction as in \cite[Lemma]{KKRW}. Let $M^{2n}$ be even-dimensional and consider the following diagram of (homotopy) fibrations
	\begin{center}
		\begin{tikzpicture}
			\node(0) at (0,0) {$B\blockdiff_\partial(D^{2n})$};
			\node(1) at (0,1.3) {$B\!\diff_\partial(D^{2n})$};
			\node(2) at (4,0) {$B\blockdiff(M)$};
			\node(3) at (4,1.3) {$B\!\diff(M)$};
			\node(4) at (0,2.6) {$\frac{\blockdiff_\partial(D^{2n})}{\diff_\partial(D^{2n})}$};
			\node(5) at (4,2.6) {$\frac{\blockdiff(M)}{\diff(M)}$};
			
			\draw[->] (1) to (0); 			
			\draw[->] (0) to (2);			
			\draw[->] (1) to (3);			
			\draw[->] (3) to (2);		
			\draw[->] (4) to (5);			
			\draw[->] (5) to (3);			
			\draw[->] (4) to (1);
		\end{tikzpicture}
	\end{center}
	If $M$ is $\ell$-connected with $\ell\le 2n-4$, then the induced map on homotopy fibres is $(2\ell-2)$-connected by Morlet's lemma of disjunction (cf. \cite[Corollary 3.2 on page 29]{blr}). Now $\pi_k(B\blockdiff_\partial(D^{2n}))\cong \pi_0(\diff_\partial(D^{2n+k-1})$ is isomorphic to the finite group of exotic spheres in dimension $(2n+k)$ and $B\!\diff_\partial(D^{2n})$ is rationally $2n-5$-connected by \cite[Theorem 4.1]{RW_UpperRange}. Therefore, $\blockdiff_\partial(D^{2n})/\diff_\partial(D^{2n})$ is rationally $(2n-5)$-connected and $\blockdiff(M)/\diff(M)$ is rationally $\min(2n-5,2\ell-2)$-connected. This implies that the map 
	\[\pi_k(B\!\diff(M))\otimes\bbQ\to \pi_k(B\blockdiff(M))\otimes\bbQ\]
	is surjective for $k\le \min(2n-4,2\ell-1)$.
\end{rem}

\subsection{Surgery theory}
Let $X$ be a simply connected manifold with boundary $\partial X$. The $\emph{structure set}$ $\calS(X,\partial X)$ of $(X,\partial X)$ (sometimes written as $\calS_\partial(X)$) is defined to be the set of equivalence classes of tuples $(W,\partial W,f)$ where $W$ is a manifold with boundary $\partial W$ and $f$ is a homotopy equivalence\footnote{Since we assume $X$ to be simply connected, every homotopy equivalence is simple and we do not need to require this in the definition.} that restricts to a diffeomorphism on the boundary. Two such tuples $(W_0,\partial W_0,f_0)$ and $(W_1,\partial W_1,f_1)$ are equivalent, if there exists a diffeomorphism $\alpha\colon W_0\to W_1$ such that $f_0=f_1\circ\alpha$.

\noindent\begin{minipage}{\textwidth}
	It is a consequence of the $h$-cobordism theorem that we have the following isomorphism (\cite[Section 3.2, pp.33]{BerglundMadsen})
	
	\[\pi_k\left(\frac{\haut(M)}{\blockdiff(M)}\right)\cong \calS_\partial(D^k\times M).\]
\end{minipage}
\noindent The main result of surgery theory is that the structure set $\calS_\partial(D^k\times M)$ fits into an exact sequence of sets known as the \emph{surgery exact sequence} (cf. \cite[Theorem 10.21 and Remark 10.22]{surgerybook}):
\begin{equation}\label{eq:ses}
	L_{k+d+1}(\bbZ)\too\calS_\partial(D^k\times M) \too \calN_\partial(D^k\times M)\overset{\sigma}{\too} L_{k+d}(\bbZ)
\end{equation}

\noindent Here, $\calN_\partial(D^k\times M)$ is the set of \emph{normal invariants} which is given by equivalence classes of tuples $(W, f,\hat f, \xi)$, where $W$ is a $d+k$-dimensional manifold with (stable) normal bundle $\nu_W$, $\xi$ is a stable vector bundle over $D^k\times M$ and $f\colon W\to D^k\times M$ is a map of degree $1$ covered by a bundle map $\hat f\colon \nu_W\to \nu_{D^k\times M}\oplus\xi$ such that $(f,\hat f)$ restricts to the identity on the boundary and the equivalence relation is given by cobordism. 

Since we only consider simply connected manifolds, the relevant $L$-groups are $4$-periodic and given by (cf. \cite[Theorem 7.96]{surgerybook}
\[L_n(\bbZ) \cong\begin{cases}
	\bbZ &\quad\text{ if } n\equiv0\;(4)\\
	\bbZ/2 &\quad\text{ if } n\equiv2\;(4)\\
	0 &\quad\text{ otherwise}
\end{cases}\]

\noindent and the map $\sigma$ in the surgery exact sequence (\ref{eq:ses}) is the so-called \emph{surgery obstruction map}, which in degrees $d+k\equiv 0\;(4)$ for simply connected $M$ is given by
\[\sigma(W,f,\hat f,\xi) = \frac18\Bigl(\sign(\underbrace{W\cup (D^k\times M)}_{\eqqcolon W'}) - \sign(S^k\times M)\Bigr) =\frac18\sign(W')\]
where $\sign$ denotes the signature (cf. \cite[Lemma 7.170, Exercise 7.188]{surgerybook}). The signature of $W'$ can be computed via Hirzebruch's signature theorem, which constructs a power series
\begin{align*}
	\calL(x_1,x_2,\dots) ={}&1 + s_1x_1 + \dots + s_ix_i + \dots + s_{i,j}x_i\cdot x_j + \dots\\
		& + s_{i_1,\dots,i_n}x_{i_1}\cdots x_{i_n}+ \dots
\end{align*}
such that $\sign(W')=\scpr{\calL(p_1(TW'),p_2(TW'),\dots),\ [W']}$. Here $p_i(TW')$ are the Pontryagin classes of $W'$. Note that $f, \hat f$, and $\xi$ can be extended  trivially to $W'$. Since the map $f$ is of degree one, evaluating the Pontryagin classes of $W'$ against $[W']$ yields the same result as evaluating the Pontryagin classes of $-\xi\oplus T(S^k\times W)$, where $-\xi$ denotes the (stable) orthogonal complement to $\xi$.

In order to further analyse $\calN_\partial(D^k\times M)$, let us define $G(n)=\{f\colon S^{n-1}\to S^{n-1}\text{ homotopy equivalence}\}$ and $BG\coloneqq \colim_{n\to\infty} BG(n)$. Note, that the index shift stems from the fact that one wants to have an inclusion $O(n)\subset G(n)$ of the orthogonal group. Analogously, let $BO\coloneqq \colim_{n\to\infty}BO(n)$. Note that $BG$ is the classifying space for stable spherical fibrations whereas $BO$ is the classifying space for stable vector bundles. The inclusion $O(n)\embeds G(n)$ induces a map $BO\to BG$ and we denote the homotopy fibre by $G/O$. By \cite[Remark 10.28]{surgerybook} there is an identification 
\[\calN_\partial(D^k\times M) \cong [S^k\wedge M_+, G/O]_*.\]
Here $M_+$ is $M$ with a disjoint base point and $\wedge$ denotes the smash product of pointed spaces given by $(X,x)\wedge(Y,y) \coloneqq (X\times Y)/(X\times\{y\}\cup\{x\}\times Y)$.
The functor $S^k\wedge{(\_)}_+$ is adjoint to the $k$-fold loop space functor $\Omega^k(\_)$ and so we get  $[S^k\wedge M_+, G/O]_* \cong [M,\Omega^k G/O]$. Now $\Omega^{k+1} BG$ is the homotopy fibre of the map $\Omega^kG/O\to \Omega^kBO$. By obstruction theory (cf. \cite[p. 418]{hatcher_at}) the obstructions to the lifting problem
\begin{center}
\begin{tikzpicture}
	\node (0) at (0,0) {$M$};
	\node (1) at (3,0) {$\Omega^k BO$};
	\node (2) at (3,1.2) {$\Omega^k G/O$};
	
	\draw[->] (0) to (1);
	\draw[->] (1) to (2);
	\draw[->, dashed] (0) to (2);	
\end{tikzpicture}
\end{center}
live in the groups $H^{i+1}(M;\pi_i(\Omega^{k+1}BG))\cong H^{i+1}(M;\pi_{k+i+1}(BG))$. The homotopy groups of $\pi_k(BG)$ are isomorphic to the shifted stable homotopy groups of spheres $\pi_{k-1}^{st}$ by \cite[p. 135]{surgerybook}. By Serre's finiteness theorem, these groups are finite for $k\ge2$ and hence all obstruction groups vanish rationally, since we assumed that $k\ge1$. Since $\maps_*(M,\Omega^k BO)$ is an $H$-space, we see that for every (pointed) map $f\colon M\to \Omega^kBO$, some multiple of $f$ can be lifted to $\Omega^k G/O$. Therefore it suffices for us to specify an element in
\[[S^k\wedge M_+, BO]_* = \widetilde{\ko}^0(S^k\wedge M_+)\]
in order to get a normal invariant. Next, consider the isomorphism given by the Pontryagin character:
\begin{align*}
	\ph(\_)\coloneqq \ch(\_\otimes\bbC)\colon \widetilde{\ko}^0(S^k\wedge M_+)\otimes\bbQ\congarrow \bigoplus_{i\ge0} \widetilde H^{4i}&(S^k\wedge M_+;\bbQ)\\
		&\cong u_k\cdot\bigoplus_{i\ge0} H^{4i-k}(M;\bbQ)
\end{align*}
for $u_k$ the cohomological fundamental class in $H^k(S^k)$. The $i$-th component of the Pontryagin character is given by
\begin{align*}
	\ph_i(\xi) &= \ch_{2i}(\xi\otimes\bbC) = \frac{1}{(2i)!}\Bigl(\bigl(-2i)c_{2i}(\xi) + f(c_1(\xi),\dots, c_{2i-1}(\xi)\bigr)\Bigr)\\
	 &= \frac{(-1)^{i+1}}{(2i-1)!} p_i(\xi) 
\end{align*}
where $f(c_1(\xi),\dots, c_{2i-1}(\xi))$ is a polynomial in Chern classes of $\xi$ homogenous of degree $2i$ which vanishes since all products in $\widetilde H^*(S^k\wedge M_+;\bbQ)$ are trivial. Hence, for any collection $(x_i)\in H^{4i-k}(M;\bbQ)$ and $(A_i)\in \bbQ$ there exists a $\lambda\in\bbZ\setminus\{0\}$ and a normal invariant $(W,f,\hat f,\xi)\in\calN_\partial(D^k\times M)$ such that
\[p_i(\xi') = (-1)^{i+1}(2i-1)!\lambda A_i\cdot u_k\cdot x_i,\]

\noindent where $u_k$ denotes the cohomological fundamental class of $S^k$. This allows us to construct a normal invariant such that the underlying stable vector bundle has prescribed Pontryagin classes, which we will do in the succeeding section.

\section{Proof of main theorem}\label{sec:main}
\subsection{Prescribing Pontryagin classes}
\noindent Let $M^d$ be as in \pref{Theorem}{thm:main} and let $m\coloneqq \frac{d+k}4$. Let $j\coloneqq \min\{i\ge1\colon p_i(TM)\not=0\in H^{4i}(M;\bbQ)\}\in\{1,\dots, \floor{\frac d4}\}$, where $p_i(TM)$ denotes the $i$-th Pontryagin class of $M$.
\begin{lem}\label{lem:construction}
	There exists a normal invariant $\eta\in\calN_\partial(D^k\times M)$ with underlying stable vector bundle $\xi\to D^k\times M$ with the following property: For $\xi'$ the extension of $\xi$ by the trivial bundle to $S^k\times M$, we have 
	\[\scpr{p_j(TM\oplus -\xi')\cdot p_{m-j}(TM\oplus -\xi'),\ [S^k\times M]}\not=0\not=\scpr{p_{m}(TM\oplus -\xi'),\ [S^k\times M]} \]
	are the only non-vanishing elementary Pontryagin numbers of $TM\oplus-\xi$, and $\sigma(\eta)=0$.
\end{lem}
\begin{proof}
	Let $u_M\in H^{4m-k}(M;\bbQ)$ denote the cohomological fundamental class of $M$. Since the cup product induces a perfect pairing 
	\[H^{4j}(M;\bbQ)\times H^{4(m-j)-k}(M;\bbQ)\to\bbQ,\]
	there exists a class $x\in H^{4(m-j)-k}(M;\bbQ)$ such that $x\cdot p_j(TM)=u_M$. By the discussion in section 2 for every $A\in\bbQ$ there exists a $\lambda\in\bbZ\setminus\{0\}$ and a normal invariant $\eta=(W,f,\hat f,\xi)$ such that the (extended) stable vector bundle $\xi'$ has only $2$ higher non-vanishing rational Pontryagin classes, namely:
	\begin{align*}
		p_0(\xi') &=1 \\
		p_{m-j}(\xi') &= -{(-1)^{m-j+1}(2m-2j)!}\lambda \cdot u_k\cdot x\\
		p_{m}(\xi') &= -{(-1)^{m+1}(2m)!}\lambda A \cdot u_k\cdot u_M
	\end{align*}
	\noindent Since $j<m$ and $p_i(TM)=0$ for all $0<i<j$, we have\footnote{Since we are only interested in rational Pontryagin classes we have $p(V\oplus W)=p(V)\cdot p(W)$.}
	\begin{align*}
		p_{n}(TM\oplus&-\xi') = \sum_{i=0}^n p_i(TM)\cdot p_{n-i}(-\xi')\\
			&=\begin{cases}
				p_m(TM) + p_j(TM)\cdot p_{m-j}(-\xi') + p_m(-\xi') &\text{ if } n=m\\
				p_{m-j}(TM) + p_{m-j}(-\xi')&\text{ if } n=m-j\\
				p_{n}(TM) &\text{ otherwise}
			\end{cases}\\
		\scpr{p_j(TM\oplus &-\xi')\cdot p_{m-j}(TM\oplus -\xi'),\ [S^k\times M]}\\
			&= \scpr{p_j(TM)\cdot (p_{m-j}(TM) + p_{m-j}(-\xi')),\ [S^k\times M]}\\
			&= \scpr{p_j(TM)\cdot {(-1)^{m-j+1}(2m-2j)!}\lambda \cdot x\cdot u_k,\ [S^k\times M]}\\
			&= {(-1)^{m-j+1}(2m-2j)!}\lambda \eqqcolon b \not=0
	\end{align*}
	\begin{align*}
		\scpr{p_m&(TM\oplus -\xi'),\ [S^k\times M]}\\
			&= \scpr{p_m(TM) + p_j(TM)\cdot (p_{m-j}(TM) + p_{m-j}(-\xi')) + p_m(-\xi'),\ [S^k\times M]}\\
			&= {(-1)^{m-j+1}(2m-2j)!}\lambda +  \scpr{{(-1)^{m+1}(2m)!}\lambda A \cdot u_k\cdot u_M,\ [S^k\times M]}\\
			&= \underbrace{(-1)^{m-j+1}\lambda(2m-2j)!}_{=b} + \underbrace{(-1)^{m+1}\lambda  (2m)!}_{\eqqcolon c\not=0}\cdot A = b + c\cdot A\
	\end{align*}
	We will choose $A$ later. Note that every non-vanishing elementary Pontryagin number of $TM\oplus -\xi$ must contain a Pontryagin class of $\xi$. Otherwise it would be a Pontryagin number of $TM$ of total degree $d+k$ which would evaluate trivially against $[S^k\times M]$. Therefore, any non-vanishing elementary Pontryagin-number of $TM\oplus-\xi$ must either contain $p_{m-j}$ or $p_m$ and since $p_i(TM)=0$ for all $0<i<j$, the above are the only possibly non-vanishing ones. It remains to compute the surgery obstruction using Hirzebruch's signature theorem:	\begin{align*}
	\sigma(\eta) &= \sign(W') = \scpr{\calL(W'),\ [W']} = \scpr{\calL(W'),\ f_*[S^k\times M]}\\
		&= \scpr{\calL(S^k)\cdot\calL(TM\oplus\xi),\ [S^k\times M]}\\
		&= \scpr{s_{j,m-j}p_j(TM\oplus\xi) p_{m-j}(TM\oplus\xi) + s_mp_m(TM),\ [S^k\times M]}\\
		&= (s_{j,m-j} + s_m)b  + s_m\cdot c \cdot A. 
	\end{align*}
	Since all coefficients $s_{\dots}$ in the $\calL$-polynomial are nonzero by \cite{BergBerg}, it follows that we can choose $A\not=0$ such that $\sigma(\eta)=0$. 
\end{proof}

\noindent By the discussion in \pref{Section}{sec:prelim}, there exists a bundle $E\to S^k$ with the same two non-vanishing elementary Pontryagin numbers and by \cite[Lemma 2.5]{a-hat-bundles} the $\hat\calA$-genus of $E$ does not vanish. Also note that $E$ is fibre homotopy equivalent to the trivial bundle.

\begin{lem}\label{lem:section}
	If $j\coloneqq \min\{i\ge1\colon p_i(TM)\not=0\}$ is smaller than $d/4$, then there exists a bundle $E\to S^k$ as above that has a cross-section with trivial normal bundle.
\end{lem}
\begin{proof}
	Let $\mathrm{triv}\colon S^k\embeds S^k\times M$ be the trivial section. Since the bundle $E$ constructed in \pref{Lemma}{lem:construction} is fibre homotopy equivalent to the trivial bundle via $f\colon S^k\times M\simeq E$ we get a section $s\coloneqq f\circ \mathrm{triv}\colon S^k\to E$. We have
	\begin{align*}
		s^*p_n(TE) & = \mathrm{triv}^*\left(\sum_{i=0}^n p_i(TM)\cdot p_{n-i}(-\xi)\right)\\
			& = \sum_{i=0}^n \underbrace{\mathrm{triv}^*p_i(TM)}_{=0 \text{ for } i\ge1}\cdot \mathrm{triv}^*p_{n-i}(-\xi) = \mathrm{triv}^* p_n(-\xi)
	\end{align*}
	Recall, that the only non-vanishing Pontryagin classes of $\xi$ are $p_{m-j}$ and $p_m$ and let $\nu_s$ denote the normal bundle of $s$. Since the rank of this bundle is bigger than $k$, the bundle $\nu_s$ is stable in the sense that it is classified by an element in
	\[
	\pi_k(BO) = \ko^{-k}(\pt)\cong \begin{cases}
		\bbZ & \text{ for } k\equiv 0\;(4)\\
		\bbZ/2 & \text{ for } k\equiv 1,2\;(8)\\
		0 & \text{ otherwise}
	\end{cases}.
	\]
	Since we are only interested in the problem rationally, it suffices to consider the case $k\equiv0\;(4)$. It follows, that $\nu_s$ is trivial if $p_{k/4}(\nu_s)=0$ and as $p(S^k)=1$, the Pontryagin class $p_{k/4}$ of $\nu_s$ satisfies
	\begin{align*}
		p_{k/4}(\nu_s) = p_{k/4}(s^*TE) = s^*p_{k/4}(TE) = \mathrm{triv}^*p_{k/4}(\xi)=0
	\end{align*}
	since by our assumption $k/4 < \frac{d+k}{4}-j = m-j$ and $p_{m-j}$ and $p_m$ are the only Pontryagin classes of $\xi$.
\end{proof}

\begin{rem}\label{rem:section}
	If $d\not\equiv0\;(4)$, the requirement from the lemma is automatically full-filled. If $d\equiv0\;(4)$ and $j=d/4$, then $M$ has only one non-vanishing Pontryagin number, namely $\scpr{p_{d/4}(TM),\ [M]}$. Since all coefficients in the $\hat\calA$-polynomial are nonzero by \cite{BergBerg}, we have $\hat\calA(M) = a\cdot\scpr{p_{d/4}(TM),\ [M]}\not=0$ for some $a\in\bbZ\setminus\{0\}$. If additionally $M$ admits a $\Spin$-structure, then by the Lichnerowicz-formula and the Atiyah--Singer index theorem \cite{atiyahsinger, lichnerowicz}, $M$ does not support a metric of positive scalar curvature. Hence, for a $\Spin$-manifold of positive scalar curvature, we have $j<d/4$ and \pref{Lemma}{lem:section} applies.
\end{rem}

\noindent From the discussion in the preceding section we get:

\begin{prop}\label{prop:bundle}
	Let $k\ge1$ and let $M$ be an oriented, simply connected manifold of dimension $d\ge\max(3k+1,2k+5)$ that has at least one non-vanishing Pontryagin class. If $d+k\equiv0\;(4)$, then there exists a smooth, oriented $M$-bundle $E\to S^k$ that is fibre homotopy equivalent to the trivial bundle and satisfies $\hat\calA(E)\not=0$. If $M$ admits a $\Spin$-structure and a metric of positive scalar curvature, then the bundle admits a cross-section with trivial normal bundle.
\end{prop}

\begin{rem}\label{rem:a-hat}
	\begin{enumerate}
		\item This recovers \cite[Theorem 1.4]{HankeSchickSteimle} and provides an upgrade: the result in loc.cit. is \enquote{based on abstract existence results [and] does not yield an explicit description of the diffeomorphism type of the fibre manifold} \cite[p. 337]{HankeSchickSteimle}. In contrast, our result states, that it is correct for generic manifolds.
		\item By \cite[Proposition 1.9]{HankeSchickSteimle} and \cite[Lemma 2.3]{Wiemeler} a bundle $M\to E\to S^k$ is rationally nullcobordant, if all rational Pontryagin classes vanish or if $\dim(M)<\frac k2$. This shows that both assumptions on $M$ from \pref{Proposition}{prop:bundle} (and hence from \pref{Theorem}{thm:main}) are actually necessary, even though the dimension bound is not be optimal (cf. \pref{Remark}{rem:morlet}). 
	\end{enumerate}
\end{rem}

\noindent Recall that an oriented manifold $M$ is called \emph{$\hat\calA$-multiplicative fibre in degree $k$} if for every oriented $M$-bundle $E\to S^K$ we have $\hat\calA(E)=0$ (cf. \cite[Definition 1.8]{HankeSchickSteimle}). From \pref{Proposition}{prop:bundle} and \pref{Remark}{rem:a-hat} (ii) we deduce the following corollary.

\begin{cor}
	A manifold $M$ of dimension $d\ge\max(3k+1,2k+5)$ is an $\hat\calA$-multiplicative fibre in degree $k$ if and only if all its rational Pontryagin classes vanish.
\end{cor}

\subsection{$\Spin$-structures and positive (scalar) curvature}
Let $M$ be $\Spin$ and let $B\!\diffs(M)$ be the classifying space for $M$-bundles with a $\Spin$-structure on the vertical tangent bundle\footnote{A model for $B\!\diffs(M)$ is given by 
\begin{align*}
	B\!\diffs(M) &\coloneqq\{(N,\hat \ell_n), M\cong N\subset\bbR^{\infty}, \hat\ell_N\in\mathrm{Bun}(TN, \theta^*U_d)\}
\end{align*}
for $\theta\colon B\!\Spin(d)\to BSO(d)$ the $2$-connected cover, $U_d\to BSO(d)$ the universal oriented vector bundle and $\mathrm{Bun}(\_,\_)$ the space of bundle maps.}. By \cite[Lemma 3.3.6]{ebert_thesis} the homotopy fibre of the forgetful map $B\!\diffs(M)\to B\!\diff(M)$ is a $K(\bbZ/2,1)$ if $M$ is simply connected. Therefore the induced map 
\[\pi_n(B\!\diffs(M))\otimes\bbQ\too\pi_n(B\!\diff(M))\otimes\bbQ\]
is an isomorphism and we may assume without loss of generality that the bundles from \pref{Section}{sec:main} carry a $\Spin$-structure on the vertical tangent bundle and hence on the total space, provided that $M$ admits one. 

It is a well known consequence of the Atiyah--Singer Index theorem and the Lichnerowicz formula that $\Spin$-manifolds with non-vanishing $\hat\calA$-genus do not admit a metric of positive scalar curvature \cite{atiyahsinger, lichnerowicz}. \pref{Theorem}{thm:main} then follows from another standard argument that goes back to Hitchin \cite{hitchin_spinors} (see \cite[Remark 1.5]{HankeSchickSteimle} or \cite[Proposition 3.7]{a-hat-bundles}) from \pref{Proposition}{prop:bundle}. Together with \cite[Theorem A]{actionofmcg} (see also \cite[Corollary E]{ownthesis}) we also derive the following result.

\begin{cor}
	Let $M$ be a simply connected $\Spin$-manifold of dimension at least $6$ that admits a metric of positive scalar curvature. Then the action 
	\[\pi_0(\diff(M)) \too\pi_0\haut(\calR_\psc(M))\]
	factors through a finite group if and only if $d\not\equiv3(4)$ or $d\equiv3(4)$ and all Pontryagin classes of $M$ vanish. 
\end{cor}

\noindent\pref{Theorem}{thm:main} also recovers \cite[Theorem 1.1 a)]{HankeSchickSteimle}:

\begin{cor}
	Let $k\ge1$ and let $N$ be a $\Spin$-manifold with $\dim(N) = d \ge \max(3k+1,2k+5)$ and $d+k\equiv0\;(4)$. Then $\pi_{k-1}(\calR_\psc(N))$ contains an element of infinite order.
\end{cor}

\begin{proof}
Let $K$ be a $K3$-surface. Then for $n\coloneqq d-4\ge2$, the manifold $K\times S^n$ satisfies the hypothesis of \pref{Theorem}{thm:main} and there is a $K\times S^n$-bundle $E\to S^k$ that has non-vanishing $\hat\calA$-genus and admits a cross section with trivial normal bundle. If $N$ is an arbitrary $\Spin$-manifold of dimension $d\ge6$, then gluing in the trivial $N\setminus D^d$-bundle along this cross section yields a $N\#(K\times S^{d-4})$-bundle over $S^k$ with non-vanishing $\hat\calA$-genus. Hence the group $\pi_{k-1}(\calR_\psc(N\#(K\times S^{d-4})))$ contains an element of infinite order. Since $N$ is cobordant to $N\#(K\times S^{d-4})$ in $\Omega_{\Spin}^d(B\pi_1(N))$, the corresponding spaces of positive scalar curvature metrics are homotopy equivalent.
\end{proof}

\begin{rem}
	A more general result without any dimension restriction has been proven by Botvinnik--Ebert--Randal-Williams \cite{berw}. The methods from loc.cit. are however not constructive and do not give a way to decide if the obtained elements arise from the orbit of the action $\diff(M)\actson\calR_\psc(M)$. 
\end{rem}

\appendix

\bigskip
\printbibliography
\bigskip

\end{document}